\documentclass[11pt]{amsart}
\usepackage{amsmath, amsthm, amssymb} 
\usepackage{mathrsfs} 
\usepackage[bookmarks=false]{hyperref}

\newtheorem{theo}{Theorem}[section]
\newtheorem*{theorem}{Theorem}
\newtheorem*{corollary}{Corollary}

\newtheorem{lem}[theo]{Lemma}
\newtheorem{prop}[theo]{Proposition}

\theoremstyle{definition}

\newtheorem{example}[theo]{Example}

\newtheorem{df}[theo]{Definition}

\newtheorem{step}{Step}
\newtheorem*{claim}{Claim}

\newcommand{\R}{\mathbf{R}}
\newcommand{\C}{\mathbf{C}}
\newcommand{\Z}{\mathbf{Z}}
\newcommand{\F}{\mathbf{F}}

\newcommand{\N}{\mathbf{N}}

\newcommand{\Ad}{\operatorname{Ad}}
\newcommand{\id}{\text{\rm id}}

\newcommand{\Inn}{\operatorname{Inn}}

\newcommand{\LL}{\operatorname{L}}
\newcommand{\Haar}{\text{\rm Haar}}
\newcommand{\dpr}{^{\prime\prime}}
\newcommand{\Prob}{\operatorname{Prob}}
\newcommand{\supp}{\operatorname{supp}}

\begin{document}

\title[${\rm II_1}$ factors with an exotic maximal amenable subalgebra]{A class of ${\rm II_1}$ factors with an exotic abelian maximal amenable subalgebra}

\begin{abstract}
We show that for every mixing orthogonal representation $\pi : \Z \to \mathcal O(H_\R)$, the abelian subalgebra $\LL(\Z)$ is maximal amenable in the crossed product ${\rm II}_1$ factor $\Gamma(H_\R)\dpr \rtimes_\pi \Z$ associated with the free Bogoljubov action of the representation $\pi$. This provides uncountably many non-isomorphic $A$-$A$-bimodules which are disjoint from the coarse $A$-$A$-bimodule and of the form $\LL^2(M \ominus A)$ where $A \subset M$ is a maximal amenable masa in a ${\rm II_1}$ factor.
\end{abstract}

\author{Cyril Houdayer }

\address{CNRS-ENS Lyon \\
UMPA UMR 5669 \\
69364 Lyon cedex 7 \\
France}

\email{cyril.houdayer@ens-lyon.fr}

\thanks{Research partially supported by ANR grants AGORA NT09-461407 and NEUMANN}

\subjclass[2010]{46L10; 46L54; 46L55; 22D25}

\keywords{Free Gaussian functor; Maximal amenable subalgebras; Asymptotic orthogonality property; Rajchman measures}

\maketitle

\section{Introduction}

A (separable) finite von Neumann algebra $P$ is {\em amenable} if there exists a norm one projection $E : \mathbf{B}(\LL^2(P)) \to P$. Connes' celebrated result \cite{connes76} shows that all finite amenable von Neumann algebras are hyperfinite \cite{mvn}.

As the amenable von Neumann algebras form a monotone class, any von Neumann algebra has maximal amenable von Neumann subalgebras. Popa exhibited in \cite{popa-amenable} the first concrete examples of maximal amenable von Neumann subalgebras in ${\rm II_1}$ factors by showing that the {\em generator} maximal abelian subalgebra (masa) in free group factors is maximal amenable. In fact, Popa  showed \cite[Lemma 2.1]{popa-amenable} that the generator masa in a free group factor satisfies the {\em asymptotic orthogonality property} (see Definition $\ref{AOP}$). He then used this property to deduce that the generator masa is maximal amenable (see \cite[Corollary 3.3]{popa-amenable}).

Subsequently in \cite[Theorem 6.2]{CFRW}, the {\em radial} masa in free group factors was shown to satisfy Popa's asymptotic orthogonality property. Since the radial masa is moreover singular by \cite[Theorem 7]{radulescu-singular}, it follows maximal amenable by \cite[Corollary 2.3]{CFRW}. The {\em cup} masa in the ${\rm II_1}$ factors associated with a planar algebra subfactor \cite{brothier-these} gives another example of maximal amenable masa.

We provide in this paper new examples of maximal amenable masas in ${\rm II_1}$ factors. Our construction is natural and consists in looking at $\LL(\Z)$ as a masa inside the crossed product $\LL(\F_\infty) \rtimes \mathbf{Z}$ where the action $\Z \curvearrowright \LL(\F_\infty)$ is a free Bogoljubov action obtained via Voiculescu's {\em free Gaussian functor} \cite{voiculescu92}.  Recall from \cite[Chapter 2]{voiculescu92} that to any separable real Hilbert space $H_\R$, one can associate a finite von Neumann algebra $\Gamma(H_\R)\dpr$ which is $\ast$-isomorphic to the free group factor $\LL(\F_{\dim H_\R})$. To any orthogonal representation $\pi : \Z \to \mathcal{O}(H_\R)$  corresponds a trace-preserving action $\sigma^\pi : \Z \curvearrowright \Gamma(H_\R)\dpr$ called the {\em free Bogoljubov action} associated with the orthogonal representation $\pi$. Our main result is the following.

\begin{theorem}
Let $G$ be a countable infinite abelian group and $\pi : G \to \mathcal O(H_\R)$ a faithful mixing orthogonal representation. Denote by $\Gamma(H_\R)\dpr \rtimes_\pi G$ the crossed product ${\rm II_1}$ factor associated with the free Bogoljubov action of $\pi$. Then $\LL(G)$ is maximal amenable in $\Gamma(H_\R)\dpr \rtimes_\pi G$.
\end{theorem}

To prove the theorem, we actually show that $\LL(G) \subset \Gamma(H_\R)\dpr \rtimes_\pi G$ satisfies Popa's asymptotic orthogonality property (see Theorem $\ref{theorem-AOP}$). Observe that when $(\pi, H_\R) = (\lambda_G, \ell^2_\R(G))$ is the left regular representation, we have 
$$\left( \LL(G) \subset \Gamma(\ell^2_\R(G))\dpr \rtimes_{\lambda_G} G \right) \cong \left( \LL(G) \subset \LL(\Z) \ast \LL(G) \right)$$
and so our theorem recovers Popa's original result \cite{popa-amenable}. The interesting feature of our theorem is that we are able to prove maximal amenability for {\em any} mixing orthogonal representation. This, in turn, will allow us to obtain new examples of maximal amenable masas.

Let $A \subset M$ be a (diffuse) masa in a separable ${\rm II_1}$ factor $M$. Write $A = \LL^\infty(Y, \nu)$ where $Y$ is a second countable compact space and $\tau | A$ is given by integration against $\nu$.  Denote by  
$$\Theta : {\rm C}(Y) \otimes {\rm C}(Y) \to \mathbf B(\LL^2(M \ominus A)) : \Theta(a \otimes b) = a \; J b^* J$$
the $\ast$-representation that encodes the $A$-$A$-bimodule structure of $\LL^2(M \ominus A)$. One can then associate to $\Theta$ a unique measure class $[\eta]$ on the Borel subsets of $Y^2$ and a multiplicity function $m : Y^2 \to \{1, \dots, \infty\}$. (We can always assume that $\eta$ is a Borel probability measure on $Y^2$ quasi-invariant under the flip $\sigma : Y^2 \to Y^2: \sigma(x, y) = (y, x)$). The triple $(Y, [\eta], m)$ is a {\em conjugacy} invariant for the masa $A \subset M$ in the following sense (see \cite[Section 3]{neshveyev-stormer}). Let $A \subset M$ and $B \subset N$ be masas in ${\rm II_1}$ factors. Then there exists a unitary $U : \LL^2(M) \to \LL^2(N)$ such that $U A U^* = B$ and $U J_M U^* = J_N$ if and only if there exists a surjective Borel isomorphism $\theta : Y_A \to Y_B$ such that $\theta_\ast [\nu_A] = [\nu_B]$, $(\theta \times \theta)_\ast [\eta_A] = [\eta_B]$ and $m_B \circ (\theta \times \theta) = m_A$ ($\eta_A$-almost everywhere). From now on, since $A$ is diffuse, we will always assume $(Y, \nu) = (\mathbf T, \Haar)$, that is, $A = \LL^\infty(\mathbf T, \Haar)$.

For the three aforementioned examples, generator \cite{popa-amenable}, radial \cite{dykema-mukherjee} and cup \cite{brothier-these} masas, the corresponding $A$-$A$-bimodule $\LL^2(M \ominus A)$ is always isomorphic to an infinite direct sum of coarse $A$-$A$-bimodules. So, in that case, the measure class $[\eta]$ is simply the class of the Haar measure on $\mathbf T^2$ and the multiplicity function $m$ equals $\infty$ $\Haar$-almost everywhere. It is then natural to ask which and how many measure classes $[\eta]$ on $\mathbf T^2$ can be concretely realized as the measure class of a maximal amenable masa $A \subset M$ in a ${\rm II_1}$ factor.

In order to answer this question, first recall that two Borel measures $\mu$ and $\nu$ on a standard Borel space $X$ are {\em singular} if there exists a Borel subset $\mathcal U \subset X$ such that $\mu(\mathcal U) = 0$ and $\nu(X \setminus \mathcal U) = 0$. We say that an inclusion of a masa in a ${\rm II_1}$ factor $A \subset M$ is {\em exotic} if for the disintegration of $\eta$ with respect to the factor map $p : (\mathbf T^2, [\eta]) \to (\mathbf T, [\Haar]) : (z, t) \to t$, that we write $\eta = \int_{\mathbf T} \eta_t {\rm d}  t$, almost every Borel measure $\eta_t$ is atomless and singular with respect to the Haar measure. When $A \subset M$ is exotic, the $A$-$A$-bimodule $\LL^2(M \ominus A)$ is {\em disjoint} from the coarse $A$-$A$-bimodule $\LL^2(A) \otimes \LL^2(A)$, in the sense that nonzero $A$-$A$-sub-bimodules of $\LL^2(M \ominus A)$ are never isomorphic to $A$-$A$-sub-bimodules of $\LL^2(A) \otimes \LL^2(A)$.

A Borel measure $\mu$ on $\mathbf T$ is {\em symmetric} if $\mu(\mathcal U) = \mu(\overline{\mathcal U})$ for all Borel subset $\mathcal U \subset \mathbf T$. To any symmetric Borel probability measure $\mu$ on $\mathbf T$, one can associate a real Hilbert space
$$H_\R^\mu = \left\{ \zeta \in \LL^2(\mathbf T, \mu) : \overline{\zeta(z)} = \zeta(\overline z) \; \mu\mbox{-almost everywhere}\right\}$$
together with an orthogonal representation
$$\pi^\mu : \Z \to \mathcal O(H_\R^\mu) : \left( \pi^\mu(n) \zeta \right)(z) = z^n \zeta(z).$$
A symmetric {\em Rajchman} measure $\mu$ on $\mathbf T$ is a symmetric Borel probability measure whose Fourier-Stieltjes coefficients $\widehat \mu(n) = \int_{\mathbf T} z^n {\rm d} \mu(z)$ converge to $0$ as $|n| \to + \infty$. Equivalently, the corresponding orthogonal representation $\pi^\mu : \Z \to \mathcal O(H_\R^\mu)$ is {\em mixing} in the sense that $\langle \pi^\mu(n)\zeta_1, \zeta_2\rangle \to 0$ as $|n| \to + \infty$ for all $\zeta_1, \zeta_2 \in H_\R^\mu$ (see e.g.\ \cite[Chapter 14, p.\ 369-371]{CFS}).

By the theorem, $\LL(\Z) \subset \Gamma(H_\R^\mu)\dpr \rtimes_{\pi^\mu} \Z$ is maximal amenable for all symmetric Rajchman measures $\mu$ on $\mathbf T$. By considering the $\LL(\Z)$-$\LL(\Z)$-bimodules $\LL^2 \left( (\Gamma(H_\R^\mu)\dpr \rtimes_{\pi^\mu} \Z) \ominus \LL(\Z) \right)$ and using a combination of results in \cite{{kahane-salem}, {kechris-louveau}}, we construct in Section $\ref{proof}$ a Borel map $\eta : 2^\N \to \Prob(\mathbf T^2) : x \mapsto \eta_x$ such that:
\begin{itemize}
\item The Borel probability measures $(\eta_x)_{x \in 2^\N}$ are pairwise singular and all singular with respect to the Haar measure on $\mathbf T^2$. 
\item The measure class $[\eta_x]$ corresponds to an $A$-$A$-bimodule of the form $\LL^2(M \ominus A)$ with $A \subset M$ maximal amenable masa in a ${\rm II_1}$ factor.
\end{itemize}
In particular, we obtain the following.

\begin{corollary}
There exists an explicit continuum $(\mathcal H(x))_{x \in 2^\N}$ of pairwise non-isomorphic 
 $A$-$A$-bimodules of the form $\LL^2(M \ominus A)$ where $A \subset M$ is an exotic maximal amenable masa in a ${\rm II_1}$ factor.
\end{corollary}

By Voiculescu's celebrated result \cite[Corollary 7.6]{voiculescu96}, the ${\rm II_1}$ factors arising in the corollary are not $\ast$-isomorphic to interpolated free group factors in the sense of \cite{{dykema94}, {radulescu94}}. Moreover, by \cite[Theorem B]{houdayer-shlyakhtenko}, these ${\rm II_1}$ factors are also {\em strongly solid} in the sense of \cite[Section 4]{ozawa-popa}, that is, the normalizer of any diffuse amenable subalgebra generates an amenable subalgebra.

\subsection*{Acknowledgments} I am grateful to R\'emi Boutonnet and Sven Raum for their useful comments and careful reading of a first draft of this paper.

\section{Preliminaries}

\subsection{Elementary facts on $\varepsilon$-orthogonality}

\begin{df}
Let $H$ be a Hilbert space, $K, L \subset H$ closed subspaces and $\varepsilon \geq 0$. We say that $K$ and $L$ are $\varepsilon$-{\em orthogonal} and write $K \perp_{\varepsilon} L$ if 
$$|\langle \xi, \eta \rangle| \leq \varepsilon \|\xi\| \|\eta\|, \forall \xi \in K, \forall \eta \in L.$$
\end{df}

Observe that when $K \perp_\varepsilon L$ with $\varepsilon < 1$, we have that $K + L$ is closed. Let $H$ be a Hilbert space and $p, q \in \mathbf B(H)$ projections. We have that $pH \perp_\varepsilon qH$ if and only if $\|p q\|_\infty \leq\varepsilon$. Therefore, whenever $pH \perp_\varepsilon qH$, for all $\xi \in H$ we get
$$
\begin{aligned}
\|p \xi\|^2 + \|q \xi\|^2 & = \|p (q \xi + (p \vee q - q) \xi)\|^2 + \|q \xi\|^2 \\
& = \|p q \xi + p(p \vee q - q) \xi \|^2 + \|q \xi\|^2 \\
& \leq \|p q \xi\|^2 + \|(p \vee q - q)\xi\|^2 + 2 \|pq \xi\| \|(p \vee q - q) \xi\| + \|q \xi\|^2 \\
& \leq (1 + \varepsilon)^2 \|(p \vee q) \xi\|^2.
\end{aligned}
$$

\begin{lem}\label{4proj}
Let $0 \leq \varepsilon < \frac12$. Let $p_1, p_2, p_3, p_4 \in \mathbf B(H)$ be projections which satisfy $p_i H \perp_{\varepsilon} p_j H$ for all $i, j \in \{1, 2, 3, 4\}$ such that $i \neq j$. We have
$$(p_1 \vee p_2) H \perp_{\delta(\varepsilon)} (p_3 \vee p_4) H$$
with $\delta(\varepsilon) = \frac{2 \varepsilon}{\sqrt{1 - \varepsilon - \sqrt{2} \, \varepsilon \sqrt{1 - \varepsilon}}}$.
\end{lem}

\begin{proof}
We first prove the following easy fact: whenever $0 \leq \varepsilon < 1$ and $q_1, q_2, q_3 \in \mathbf B(H)$ are projections which satisfy $q_1 H \perp_\varepsilon q_2 H$, $q_2 H \perp_\varepsilon q_3 H$, $q_3 H \perp_\varepsilon q_1 H$, we have $(q_1 \vee q_2) H \perp_{\varepsilon'} q_3 H$ with $\varepsilon' = \frac{\sqrt{2} \, \varepsilon}{\sqrt{1 - \varepsilon}}$. Indeed, let $\xi_i \in q_i H$ for $i = 1, 2$. We have 
$$\|q_3 (\xi_1 + \xi_2)\|^2 \leq 2(\|q_3 \xi_1\|^2 + \|q_3 \xi_2\|^2) \leq 2 \varepsilon^2 (\|\xi_1\|^2 + \|\xi_2\|^2).$$
We moreover have
$$\|\xi_1 + \xi_2\|^2 \geq \|\xi_1\|^2 + \|\xi_2\|^2 - 2 \varepsilon \|\xi_1\| \|\xi_2\| \geq (1 - \varepsilon) (\|\xi_1\|^2 + \|\xi_2\|^2).$$
Altogether, we get
$$\|q_3 (\xi_1 + \xi_2)\|^2 \leq \frac{2 \varepsilon^2}{1 - \varepsilon} \|\xi_1 + \xi_2\|^2.$$

Let now $0 \leq \varepsilon < \frac12$. Applying the fact, we get $(p_1 \vee p_2) H \perp_{\varepsilon'} p_3 H$ and $(p_1 \vee p_2) H \perp_{\varepsilon'} p_4 H$ with $\varepsilon' = \frac{\sqrt{2} \, \varepsilon}{\sqrt{1 - \varepsilon}} < 1$. Applying once more the fact, we get $(p_1 \vee p_2) H \perp_{\varepsilon''} (p_3 \vee p_4) H$ with $\varepsilon'' = \frac{\sqrt{2} \, \varepsilon'}{\sqrt{1 - \varepsilon'}} = \frac{2 \varepsilon}{\sqrt{1 - \varepsilon - \sqrt{2} \, \varepsilon \sqrt{1 - \varepsilon}}}$.
\end{proof}

Write $\delta : [0, \frac12) \to \R_+ : \varepsilon \mapsto \frac{2 \varepsilon}{\sqrt{1 - \varepsilon - \sqrt{2} \, \varepsilon \sqrt{1 - \varepsilon}}}$ for the function which appears in Lemma $\ref{4proj}$.

\begin{prop}\label{projections}
Let $k \geq 1$. Let $0 \leq \varepsilon < 1$ such that $\delta^{\circ (k - 1)}(\varepsilon) < 1$. For $1 \leq i \leq 2^k$, let $p_i \in \mathbf B(H)$ be projections such that $p_i H \perp_\varepsilon p_j H$ for all $i, j \in \{1, \dots , 2^k\}$ such that $i \neq j$. Write $P_\ell = \bigvee_{i = 1}^{2^\ell} p_i$ for $1 \leq \ell \leq k$. Then for all $1 \leq \ell \leq k$ and all $\xi \in H$, we have
$$\sum_{i = 1}^{2^\ell} \|p_i \xi\|^2 \leq \prod_{j = 0}^{\ell - 1} \left( 1 + \delta^{\circ j}(\varepsilon) \right)^2 \|P_\ell \xi\|^2.$$
\end{prop}

\begin{proof}
We prove the result by induction on $k \geq 1$. It is clear for $k = 1$ as we observed above. Assume it is true for $k - 1 \geq 1$. Write $q_i = p_{2i - 1} \vee p_{2i}$ for all $i \in \{1, \dots , 2^{k - 1}\}$. By Lemma $\ref{4proj}$, we have $q_i H \perp_{\delta(\varepsilon)} q_j H$ for all $i, j \in \{1, \dots, 2^{k - 1}\}$ such that $i \neq j$. Observe that $\bigvee_{i = 1}^{2^{k - 1}} q_i = P_k$. Since $\delta^{\circ (k - 2)}(\delta(\varepsilon)) = \delta^{\circ (k - 1)}(\varepsilon) < 1$, the induction hypothesis yields 
$$\sum_{i = 1}^{2^{k - 1}} \|q_i \xi\|^2 \leq \prod_{j = 0}^{k - 2} \left(1 + \delta^{\circ j}(\delta(\varepsilon)) \right)^2 \|P_k \xi\|^2 =  \prod_{j = 1}^{k - 1} \left(1 + \delta^{\circ j}(\varepsilon) \right)^2 \|P_k \xi\|^2$$
for all $\xi \in H$. Since moreover, we have
$$\|p_{2i - 1} \xi\|^2 + \|p_{2i} \xi\|^2 \leq (1 + \varepsilon)^2 \|q_i \xi\|^2$$
for all $i \in \{1, \dots, 2^{k - 1}\}$ and all $\xi \in H$, it follows that
$$\sum_{i = 1}^{2^k} \|p_i \xi\|^2 \leq (1 + \varepsilon)^2 \sum_{i = 1}^{2^{k - 1}} \|q_i \xi\|^2 \leq  \prod_{j = 0}^{k - 1} \left( 1 + \delta^{\circ j}(\varepsilon) \right)^2 \|P_k \xi\|^2, \forall \xi \in H.$$
\end{proof}

\subsection{Voiculescu's free Gaussian functor \cite{{voiculescu85}, {voiculescu92}}}

Let $H_\R$ be a real separable Hilbert space. Let $H = H_\R \otimes_\R \C = H_\R \oplus {\rm i} H_\R$ be the corresponding complexified Hilbert space. The canonical complex conjugation on $H$ will be simply denoted by $\overline{e + {\rm i} f} = e - {\rm i} f$ for all $e, f \in H_\R$. The \emph{full Fock space} of $H$ is defined by
\begin{equation*}
\mathcal{F}(H) =\C\Omega \oplus \bigoplus_{n \geq 1} H^{\otimes n}.
\end{equation*}
The unit vector $\Omega$ is called the \emph{vacuum vector}. For all $e \in H$, we define the \emph{left creation operator}
$$
\ell(e) : \mathcal{F}(H) \to \mathcal{F}(H) : \left\{ 
{\begin{array}{l} \ell(e)\Omega = e \\ 
\ell(e)(e_1 \otimes \cdots \otimes e_n) = e \otimes e_1 \otimes \cdots \otimes e_n.
\end{array}} \right.
$$
We have $\ell(e)^* \ell(f) = \langle e, f\rangle$ for all $e, h \in H$. In particular, $\ell(e)$ is an isometry for all unit vectors $e \in H$.

For all $e \in H_\R$, put $W(e) := \ell(e) + \ell(e)^*$. Voiculescu's result \cite[Lemma 2.6.3]{voiculescu92} shows that the distribution of the selfadjoint operator $W(e)$ with respect to the vacuum vector state $\langle \cdot \Omega, \Omega\rangle$ is the semicircular law supported by the interval $[-2 \| e \|, 2 \| e \|]$. Moreover, \cite[Lemma 2.6.6]{voiculescu92} shows that for any subset $\Xi \subset H_\R$ of pairwise orthogonal vectors, the family $\{W(e) : e \in \Xi\}$ is freely independent.

We denote by $\Gamma(H_\R)$ the C$^*$-algebra generated by $\{W(e) : e \in H_\R\}$ and $\Gamma(H_\R)\dpr$ the von Neumann algebra generated by $\Gamma(H_\R)$. The vector state $\tau = \langle \cdot \Omega, \Omega\rangle$ is a faithful normal trace on $\Gamma(H_\R)\dpr$ and we have that $\Gamma(H_\R)\dpr$ is $\ast$-isomorphic to the free group factor on $\dim H_\R$ generators, that is, $\Gamma(H_\R)\dpr \cong \LL(\F_{\dim H_\R})$.

Since the vacuum vector $\Omega$ is separating and cyclic for $\Gamma(H_\R)\dpr$, any
$x \in \Gamma(H_\R)\dpr$ is uniquely determined by $\xi = x \Omega \in \mathcal{F}(H)$. Thus we will write $x = W(\xi)$. Note that for $e \in H_\R$, we recover the semicircular random variables $W(e) = \ell(e) + \ell(e)^*$ generating $\Gamma(H_\R)\dpr$. More generally we have $W(e) = \ell(e) + \ell(\overline e)^*$ for all $e \in H$. Given any vectors $e_i \in H$, it is easy to check that $e_1\otimes \cdots \otimes e_n$ lies in $\Gamma(H_\R)\dpr \Omega$. The corresponding words $W(e_1 \otimes \cdots \otimes e_n) \in \Gamma(H_\R)\dpr$ enjoy useful properties that are summarized in the following.

\begin{prop}\label{wick}
Let $e_i \in H$, for $i \geq 1$. The following are true:
\begin{enumerate}
\item We have the {\em Wick formula}:
$$W(e_1 \otimes \cdots \otimes e_n) = \sum_{k = 0}^n \ell(e_1) \cdots \ell(e_k) \ell(\overline e_{k + 1})^* \cdots \ell(\overline e_n)^*.$$
\item  If  $\langle \overline e_r, e_{r + 1}\rangle = 0$ then we have
$$W(e_1 \otimes \cdots \otimes e_r) W(e_{r + 1} \otimes \cdots \otimes e_n) = W(e_1 \otimes \cdots \otimes e_r \otimes e_{r + 1} \otimes \cdots \otimes e_n).$$
\item We have $W(e_1 \otimes \cdots \otimes e_n)^* = W(\overline e_n \otimes \cdots \otimes \overline e_1)$.
\item The linear span of $\{1, W(e_1 \otimes \cdots \otimes e_n) : n \geq 1, e_i \in H \}$ forms a unital weakly dense $\ast$-subalgebra of $\Gamma(H_\R)\dpr$.
\end{enumerate}
\end{prop}

\begin{proof}
The proof of $(1)$ is borrowed from \cite[Lemma 3.2]{houdayer-ricard}. We prove the formula by induction on $n$. For $n \in \{0, 1\}$, we have $W(\Omega)=1$ and we already observed that $W(e_i)=\ell(e_i) + \ell(\overline e_i)^*$.

Next, for $e_{0}\in H$, we have 
$$
\begin{aligned}
W(e_{0})W(e_1\otimes \cdots \otimes e_n)\Omega & = W(e_{0})(e_1\otimes \cdots \otimes e_n)
\\ 
& = (\ell(e_{0})+\ell(\overline e_{0})^*)e_1\otimes \cdots \otimes e_n \\ 
& = e_{0}\otimes e_1\otimes \cdots \otimes e_n + \langle \overline e_{0},e_{1}\rangle \, e_2\otimes \cdots \otimes e_n.
\end{aligned}
$$
So, we obtain 
$$
\begin{aligned}
W(e_0\otimes \cdots \otimes e_n) & =W(e_{0})W(e_1\otimes \cdots \otimes e_n) - \langle \overline e_{0},e_{1}\rangle W(e_2\otimes \cdots \otimes e_n) \\
& = \ell(\overline e_{0})^*W(e_1\otimes \cdots \otimes e_n) - \langle \overline e_{0},e_{1}\rangle W(e_2\otimes \cdots \otimes e_n) \\
& \ \ \ + \ell(e_{0}) W(e_1\otimes \cdots \otimes e_n).
\end{aligned}
$$
Using the assumption for $n$ and $n-1$ and the relation $\ell(\overline e_0)^*\ell(e_1) = \langle \overline e_0, e_1\rangle$, we obtain
$$\ell(\overline e_{0})^*W(e_1\otimes \cdots \otimes e_n)=\langle \overline
e_{0},e_{1}\rangle W(e_2\otimes \cdots \otimes e_n) + \ell(\overline
e_{0})^*\ell(\overline e_{1})^* \cdots \ell(\overline e_{n})^*.$$ 
Since $\ell(e_{0})W(e_1\otimes \cdots \otimes e_n)$ gives the last $n + 1$ terms in
the Wick formula at order $n+1$ and $\ell(\overline e_{0})^*\ell(\overline e_{1})^* \cdots \ell(\overline e_{n})^*$ gives the first term, we are done.

$(2)$ By the Wick formula, we have that $W(e_1 \otimes \cdots \otimes e_r) W(e_{r + 1} \otimes \cdots \otimes e_n)$ is equal to 
$$\sum_{0 \leq j \leq r \leq k \leq n} \ell(e_1) \cdots \ell(e_j) \ell(\overline e_{j + 1})^* \cdots \ell(\overline e_r)^* \ell(e_{r + 1}) \cdots \ell(e_k) \ell(\overline e_{k + 1})^* \cdots \ell(\overline e_n)^*.
$$
Whenever $j \leq r - 1$ and $k \geq r + 1$, since $\langle \overline e_r, e_{r + 1} \rangle = 0$, we have 
$$\ell(e_1) \cdots \ell(e_j) \ell(\overline e_{j + 1})^* \cdots \ell(\overline e_r)^* \ell(e_{r + 1}) \cdots \ell(e_k) \ell(\overline e_{k + 1})^* \cdots \ell(\overline e_n)^* = 0.$$
Therefore the above sum simply equals
$$
\sum_{0 \leq j \leq r - 1} \ell(e_1) \cdots \ell(e_j) \ell(\overline e_{j + 1})^* \cdots \ell(\overline e_n)^* + \sum_{r \leq k \leq n} \ell(e_1) \cdots \ell(e_k) \ell(\overline e_{k + 1})^* \cdots \ell(\overline e_n)^*
$$
and so $W(e_1 \otimes \cdots \otimes e_r) W(e_{r + 1} \otimes \cdots \otimes e_n) = W(e_1 \otimes \cdots \otimes e_r \otimes e_{r + 1} \otimes \cdots \otimes e_n)$.

$(3)$ It is a straightforward consequence of $(1)$. 

$(4)$ Denote by $\mathcal W$ the linear span of $\{1, W(e_1 \otimes \cdots \otimes e_n) : n \geq 1, e_i \in H \}$. We only have to show that $\mathcal W$ is stable under taking products. Let $e_0, \dots, e_m$, $f_1, \dots, f_n \in H$. We prove by induction on $m$ that $W(e_0 \otimes \cdots \otimes e_m) W(f_1 \otimes \cdots \otimes f_n) \in \mathcal W$. As we observed above, we have
$$ W(e_{0})W(f_1\otimes \cdots \otimes f_n) = W(e_0\otimes f_1 \otimes \cdots \otimes f_n)  + \langle \overline e_{0}, f_{1} \rangle W(f_2\otimes \cdots \otimes f_n) \in \mathcal W$$
so the result is true for $m = 0$. Assume it is true for all $0 \leq k \leq m - 1$. We can write $W(e_0 \otimes \cdots \otimes e_m) W(f_1 \otimes \cdots \otimes f_n)$ as
$$W(e_0) W(e_1 \otimes \cdots \otimes e_m) W(f_1 \otimes \cdots \otimes f_n) - \langle \overline e_0, e_1 \rangle W(e_2 \otimes \cdots \otimes e_m) W(f_1 \otimes \cdots \otimes f_n).$$
Using the induction hypothesis, we get that $W(e_0 \otimes \cdots \otimes e_m) W(f_1 \otimes \cdots \otimes f_n) \in \mathcal W$. This shows that $\mathcal W$ is a unital weakly dense $\ast$-subalgebra of $\Gamma(H_\R)\dpr$.
\end{proof}

Let $G$ be a countable group together with an orthogonal representation $\pi : G \to \mathcal{O}(H_\R)$. We shall still denote by $\pi : G \to \mathcal{U}(H)$ the corresponding unitary representation on the complexified Hilbert space $H = H_\R \otimes_\R \C$. The {\it free Bogoljubov action} $\sigma^\pi : G \curvearrowright (\Gamma(H_\R)\dpr, \tau)$ associated with the representation $\pi$ is defined by
$$
\sigma_g^\pi = \Ad(\mathcal{F}(\pi(g))), \forall g \in G,
$$
where $\mathcal{F}(\pi(g)) = \id_{\C \Omega} \oplus \bigoplus_{n \geq 1} \pi(g)^{\otimes n} \in \mathcal{U}(\mathcal{F}(H))$.

\begin{example}
If $(\pi, H_\R) = (\lambda_G, \ell_\R^2(G))$ is the left regular orthogonal representation of $G$, then the action $\sigma^{\lambda_G} : G \curvearrowright \Gamma(\ell^2_\R(G))\dpr$ is the free Bernoulli shift and in that case we have 
$$\left( \LL(G) \subset \Gamma(\ell^2_\R(G))\dpr \rtimes_{\lambda_G} G \right)  \cong \left( \LL(G) \subset \LL(\Z) \ast \LL(G) \right).$$
\end{example}

Recall that an orthogonal representation $\pi : G \to \mathcal O(H_\R)$ is {\em mixing} if $\lim_{g \to \infty} \langle \pi(g)\xi, \eta \rangle = 0$ for all $\xi, \eta \in H_\R$.

\begin{prop}
Let $G$ be a countable group together with an orthogonal representation $\pi : G \to \mathcal{O}(H_\R)$. The following are equivalent:
\begin{enumerate}
\item The representation $\pi : G \to \mathcal{O}(H_\R)$ is mixing.
\item The $\tau$-preserving action $\sigma^\pi : G \curvearrowright \Gamma(H_\R)\dpr$ is mixing, that is, 
\begin{equation*}
\lim_{g \to \infty} \tau(\sigma_g^\pi(x) y) = 0, \forall x, y \in \Gamma(H_\R)\dpr \ominus \C.
\end{equation*}
\end{enumerate}
\end{prop}

\begin{proof}
$(1) \Rightarrow (2)$. Observe that since the linear span of 
$$\{1, W(e_1 \otimes \cdots \otimes e_n) : n \geq 1, e_i \in H\}$$
is a unital weakly dense $\ast$-subalgebra of $\Gamma(H_\R)\dpr$, it suffices to show that $\tau(\sigma_g^\pi(x)y) \to 0$ as $g \to \infty$ for $x = W(e_1 \otimes \cdots \otimes e_m)$, $y = W(f_1 \otimes \cdots \otimes f_n)$. Using Proposition $\ref{wick}$ $(2)$, for all $g \in G$, we get
$$\tau(\sigma_g^\pi(x) y) = \frac{\langle \pi(g)e_m, \overline f_1 \rangle}{\|\overline  f_1\|^2} \tau \left( W(\pi(g)e_1 \otimes \cdots \otimes \pi(g)e_{m - 1} \otimes \overline f_1) y \right).$$
Since $\pi$ is mixing, we obtain $\lim_{g \to \infty} \tau(\sigma_g^\pi(x) y)$.

$(2) \Rightarrow (1)$. Let $e, f \in H_\R$. Using Proposition $\ref{wick}$ $(2)$, for all $g \in G$, we get
$$\lim_{g \to \infty} \langle \pi(g)e, f \rangle = \lim_{g \to \infty} \tau \left( \sigma_g^\pi(W(e)) W(f) \right) = 0.$$
\end{proof}

As a consequence of the previous proposition and \cite[Theorem 3.1]{popa-malleable1}, we obtain that whenever $\pi : G \to \mathcal O(H_\R)$ is a mixing representation of an abelian group $G$, $\LL(G)$ is a {\em singular} masa in $\Gamma(H_\R)\dpr \rtimes_\pi G$, that is, $$\{u \in \mathcal U(\Gamma(H_\R)\dpr \rtimes_\pi G) : u \LL(G) u^* = \LL(G)\} = \mathcal U(\LL(G)).$$

Finally, recall from \cite[Theorem 5.1]{houdayer-shlyakhtenko} that whenever the orthogonal representation $\pi : G \to \mathcal O(H_\R)$ is faithful, the associated free Bogoljubov action $\sigma^\pi : G \curvearrowright \Gamma(H_\R)\dpr$ is {\em properly outer}, that is, $\sigma_g^\pi \notin \Inn(\Gamma(H_\R)\dpr)$ for all $g \in G \setminus \{1\}$. In that case, we have
$$\Gamma(H_\R)' \cap (\Gamma(H_\R)\dpr \rtimes_\pi G) = \Gamma(H_\R)' \cap \Gamma(H_\R)\dpr = \C$$
and so $\Gamma(H_\R)\dpr \rtimes_\pi G$ is a ${\rm II_1}$ factor.

\section{The asymptotic orthogonality property}

We refer to \cite[Appendix A]{BO} for a brief account on ultrafilters and ultraproducts of tracial von Neumann algebras.

\begin{df}[\cite{popa-amenable}]\label{AOP}
Let $(M, \tau)$ be a tracial von Neumann algebra. We say that a von Neumann subalgebra $A \subset M$ has the {\em asymptotic orthogonality property} if there exists a free ultrafilter $\omega$ on $\N$ such that for all $x, y \in (M^\omega \ominus A^\omega) \cap A'$ and all $a, b \in M \ominus A$, the vectors $ax$ and $yb$ are orthogonal in $\LL^2(M^\omega, \tau_\omega)$.
\end{df}

Popa proved in \cite[Lemma 2.1]{popa-amenable} that the generator masa in free group factors satisfies the asymptotic orthogonality property. The main result of this section is the following.

\begin{theo}\label{theorem-AOP}
Let $G$ be an infinite countable abelian group and $\pi : G \to \mathcal O(H_\R)$ a mixing orthogonal representation. Denote by $\Gamma(H_\R)\dpr \rtimes_\pi G$ the crossed product von Neumann algebra. Then $\LL(G) \subset \Gamma(H_\R)\dpr \rtimes_\pi G$ has the asymptotic orthogonality property.
\end{theo}

\begin{proof}
We denote by $H = H_\R \otimes_\R \C$ the complexification of $H_\R$ and $\mathcal H = \mathcal F(H)$ the full Fock space of $H$. The conjugation on $H$ is simply denoted by $e \mapsto \overline e$. We still denote by $\pi : G \to \mathcal U(H)$ the corresponding unitary representation. Observe that $\pi(g)\overline e = \overline{\pi(g)e}$ for all $g \in G$ and all $e \in H$. Let $\sigma : G \curvearrowright \Gamma(H_\R)\dpr$ be the free Bogoljubov action associated with $\pi$ and $\rho : G \to \mathcal U(\mathcal H)$ the Koopman representation of the action $\sigma$. Observe that 
$$\rho(g) = \id_{\C \Omega} \oplus \bigoplus_{n \geq 1} \pi(g)^{\otimes n}, \forall g \in G.$$ 
Put $Q = \Gamma(H_\R)\dpr$ and $M = \Gamma(H_\R)\dpr \rtimes_\pi G$. We will always identify the GNS-Hilbert space $\LL^2(Q)$ with the full Fock space $\mathcal H$ via the unitary operator
$$U : \LL^2(Q) \to \mathcal H : 
\left\{ 
{\begin{array}{l} 1 \mapsto \Omega \\ 
W(e_1 \otimes \cdots \otimes e_n) \mapsto e_1 \otimes \cdots \otimes e_n.
\end{array}} \right.
$$
We denote by $J : \mathcal H \to \mathcal H$ the canonical conjugation
$$J \Omega = \Omega \; \mbox{ and } \; J(e_1 \otimes \cdots \otimes e_n) = \overline e_n \otimes \cdots \otimes \overline e_1.$$
We will identify $\LL^2(M)$ with $\mathcal H \otimes \ell^2(G)$ via the unitary operator $\LL^2(M) \ni a u_h \mapsto U(a) \otimes \delta_h \in \mathcal H \otimes \ell^2(G)$. We denote by $\mathcal J : \mathcal H \otimes \ell^2(G) \to \mathcal H \otimes \ell^2(G)$ the conjugation
$\mathcal J (\xi \otimes \delta_g) = J \rho(g^{-1})\xi \otimes \delta_{g^{-1}}$.

With a proof that is very similar to \cite[Lemma 2.1]{popa-amenable}, we will reach the conclusion of Theorem $\ref{theorem-AOP}$. We fix once and for all a free ultrafilter $\omega$ on $\N$. We want to show that $a x \perp y b$ for all $x, y \in (M^\omega \ominus \LL(G)^\omega) \cap \LL(G)'$ and all $a, b \in M \ominus \LL(G)$. Note that since $G$ is abelian, we have $x u_g \in (M^\omega \ominus \LL(G)^\omega) \cap \LL(G)'$ whenever $x \in (M^\omega \ominus \LL(G)^\omega) \cap \LL(G)'$. So, using Kaplansky density theorem, it suffices to show that $a x \perp y b$ for all $x, y \in (M^\omega \ominus \LL(G)^\omega) \cap \LL(G)'$ and $a, b$ of the form $a = W(\xi_1 \otimes \cdots \otimes \xi_s)$, $b = W(\eta_1 \otimes \cdots \otimes \eta_t)$.

From now on, we fix $a = W(\xi_1 \otimes \cdots \otimes \xi_s)$ and $b = W(\eta_1 \otimes \cdots \otimes \eta_t)$. Denote by $K \subset H$ the finite dimensional subspace generated by $\xi_i$ and  $\eta_j$ and write $r = \dim K$. We will further assume that $\|\xi_i\| = \|\eta_j\| = 1$ for all $i, j$ and that $K = \overline K$. 

Fix $h \in G$. Denote by $\mathcal X_h \subset \mathcal H \ominus (\C \Omega \oplus H)$ the closed linear span of all the words $e_1 \otimes \cdots \otimes e_n$ where $n \geq 2$ and such that the first letter $e_1$ belongs to $K$ or the last letter $e_n$ belongs to $\pi(h)K$. Using the above identification, we can then split $\mathcal X_h$ as an orthogonal sum $\mathcal X_h = \mathcal X_h^1 \oplus \mathcal X_h^2 \oplus \mathcal X_h^3$ such that
$$
\begin{aligned}
\mathcal X_h^1 & = K \otimes \mathcal H \otimes \pi(h)K \\
\mathcal X_h^2 & = K \otimes \mathcal H \otimes (\pi(h)K)^\perp \\
\mathcal X_h^3 & = K^\perp \otimes \mathcal H \otimes \pi(h)K.
\end{aligned}
$$
Likewise, denote by $\mathcal Y_h \subset \mathcal H \ominus (\C \Omega \oplus H)$ the closed linear span of all the words $e_1 \otimes \cdots \otimes e_n$ where $n \geq 2$ and such that the first letter $e_1$ belongs to $K^\perp$ and the last letter $e_n$ belongs to $(\pi(h)K)^\perp$, that is, $\mathcal Y_h = K^\perp \otimes \mathcal H \otimes (\pi(h)K)^\perp$. Therefore we have 
$$\mathcal H \ominus \C \Omega = K \oplus K^\perp \oplus \mathcal X_h \oplus \mathcal Y_h = K \oplus K^\perp \oplus \mathcal X_h^1 \oplus \mathcal X_h^2 \oplus \mathcal X_h^3 \oplus \mathcal Y_h.$$

\begin{claim}
For every $\varepsilon > 0$, there exists a finite subset $\mathcal F_\varepsilon \subset G$ such that 
$$\rho(g)K \perp_\varepsilon K \; \mbox{ and } \; \rho(g) \mathcal X_h^i \perp_\varepsilon \mathcal X_h^i$$
for all $g \in G \setminus \mathcal F_\varepsilon$, all $i \in \{1, 2, 3\}$ and all $h \in G$.
\end{claim}

\begin{proof}[Proof of the Claim]
Fix $\varepsilon > 0$. Let $\zeta_1, \dots, \zeta_r$ be an orthonormal basis of $K$. Since $\pi$ is a mixing representation, there exists a finite subset $\mathcal F_\varepsilon \subset G$ such that $|\langle \pi(g)\zeta_i , \zeta_j \rangle| \leq \varepsilon / r$ for all $g \in G \setminus \mathcal F_\varepsilon$ and all $i , j \in \{1, \dots, r\}$. Observe that since $G$ is abelian, we also have 
\begin{equation}\label{equation-abelian}
| \langle \pi(g)\pi(h)\zeta_i , \pi(h)\zeta_j \rangle | \leq \frac{\varepsilon}{r}, \forall g \in G \setminus \mathcal F_\varepsilon, \forall h \in G, \forall i, j \in \{1, \dots, r\}. 
\end{equation}

Let $\xi, \eta \in K \otimes \mathcal H$ that we write $\xi = \sum_{i = 1}^r \zeta_i \otimes e_i$ and $\eta = \sum_{j = 1}^r \zeta_j \otimes f_j$, with $e_i, f_j \in \mathcal H$. Note that $\|\xi\|^2 = \sum_{i = 1}^r \|e_i\|^2$ and $\|\eta\|^2 = \sum_{j = 1}^r \|f_j\|^2$. We have $\rho(g)\xi = \sum_{i = 1}^r \pi(g)\zeta_i \otimes \rho(g)e_i$. Thus, for all $g \in G \setminus \mathcal F_\varepsilon$, using Cauchy-Schwarz inequality, we get
$$
|\langle \rho(g) \xi, \eta \rangle| \leq \sum_{i, j = 1}^r |\langle \pi(g) \zeta_i, \zeta_j\rangle| |\langle \rho(g) e_i, f_j \rangle| \leq \frac{\varepsilon}{r} \sum_{i, j = 1}^r \|e_i\| \|f_j\| \leq \varepsilon \|\xi\| \|\eta\|.
$$
So, we obtain $\rho(g)(K \otimes \mathcal H) \perp_\varepsilon K \otimes \mathcal H$ and, in particular, $\rho(g)K \perp_\varepsilon K$, $\rho(g) \mathcal X_h^1 \perp_\varepsilon \mathcal X_h^1$, $\rho(g) \mathcal X_h^2 \perp_\varepsilon \mathcal X_h^2$ for all $g \in G \setminus \mathcal F_\varepsilon$ and all $h \in G$.

Likewise, using Inequality $(\ref{equation-abelian})$, we obtain $\rho(g) (\mathcal H \otimes \pi(h)K) \perp_\varepsilon \mathcal H \otimes \pi(h)K$ and, in particular, $\rho(g) \mathcal X_h^3 \perp_\varepsilon \mathcal X_h^3$ for all $g \in G \setminus \mathcal F_\varepsilon$ and all $h \in G$. 
\end{proof}

Let $x \in (M^\omega \ominus \LL(G)^\omega) \cap \LL(G)'$. We may and will always represent $x$ by a sequence $(x_n)$ such that $\sup_n \|x_n\|_\infty \leq 1$; $x_n \in M \ominus \LL(G)$; and $\lim_{n \to \omega} \|u_g x_n u_g^* - x_n\|_2 = 0$ for all $g \in G$. Write $x_n = \sum_{h \in G} (x_n)^h u_h$ for the Fourier expansion of $x_n$ in $M$ with respect to the crossed product decomposition $M = Q \rtimes G$. Observe that $(x_n)^h \in Q \ominus \C$ for all $n \in \N$ and all $h \in G$. Define subspaces of $(\mathcal H \ominus \C \Omega) \otimes \ell^2(G)$ by $\mathscr X = \bigoplus_{h \in G} (\mathcal X_h \otimes \C \delta_h)$ and $\mathscr Y = \bigoplus_{h \in G} (\mathcal Y_h \otimes \C \delta_h)$. Under the previous identification, we then have
\begin{equation}\label{module}
\LL^2(M \ominus \LL(G)) = (K \otimes \ell^2(G)) \oplus (K^\perp \otimes \ell^2(G)) \oplus \mathscr X \oplus \mathscr Y.
\end{equation}

\begin{step}\label{step1}
For all $x = (x_n) \in (M^\omega \ominus \LL(G)^\omega) \cap \LL(G)'$, we have 
$$\lim_{n \to \omega} \|P_{(K \otimes \ell^2(G)) \oplus \mathscr{X}}(x_n)\|_2 = 0.$$
\end{step}

\begin{proof}[Proof of Step $\ref{step1}$]
We will be using the notation $\mathcal X^0_h := K$ for all $h \in G$. For all $g, h \in G$, all $i \in \{0, 1, 2, 3\}$ and all $n \in \N$, we have
$$
\begin{aligned}
\| P_{\mathcal X_h^i} ((x_n)^h) \|_2^2 & = \|\rho(g) P_{\mathcal{X}_h^i} ((x_n)^h) \|_2^2 \\
& = \|\rho(g) P_{\mathcal X_h^i} ((x_n)^h) - P_{\rho(g) \mathcal{X}_h^i} ((x_n)^h) +  P_{\rho(g) \mathcal{X}_h^i} ((x_n)^h)\|_2^2 \\
& \leq 2\|\rho(g) P_{\mathcal X_h^i} ((x_n)^h) - P_{\rho(g) \mathcal{X}_h^i} ((x_n)^h)\|_2^2 + 2\|P_{\rho(g) \mathcal{X}_h^i} ((x_n)^h)\|_2^2 \\
& = 2\| P_{\rho(g)\mathcal{X}_h^i}( u_g (x_n)^h u_g^* - (x_n)^h) \|_2^2 + 2\|P_{\rho(g) \mathcal{X}_h^i} ((x_n)^h)\|_2^2 \\
& \leq 2\| \sigma_g ((x_n)^h) - (x_n)^h \|_2^2 + 2\|P_{\rho(g) \mathcal{X}_h^i} ((x_n)^h)\|_2^2.
\end{aligned}
$$

Fix $k \geq 1$. Choose $\varepsilon > 0$ very small such that  $\prod_{\ell = 0}^{k - 1} (1 + \delta^{\circ \ell}(\varepsilon))^2 \leq 2$, where $\delta : [0, \frac12) \to \R$ is the function which appeared in Lemma $\ref{4proj}$. Then choose a finite subset $\mathcal F_\varepsilon \subset G$ according to the Claim. Finally, choose a subset $\mathcal G \subset G$ of cardinality $2^k$ with the property that $s^{-1} t \in G \setminus \mathcal F_\varepsilon$ whenever $s, t \in \mathcal G$ such that $s \neq t$. So, we have that $\rho(s)\mathcal X_h^i \perp_\varepsilon \mathcal \rho(t) \mathcal X_h^i$ for all $s, t \in \mathcal G$ such that $s \neq t$, all $i \in \{0, 1, 2, 3\}$ and all $h \in G$. Therefore, using Proposition $\ref{projections}$ and the above inequality, we get
$$
\begin{aligned}
2^k \| P_{\mathcal{X}_h^i} ((x_n)^h) \|_2^2 & =  \sum_{g \in \mathcal{G}} \| \rho(g) P_{\mathcal{X}_h^i} ((x_n)^h) \|_2^2 \\
& \leq  \sum_{g \in \mathcal{G}} \left( 2\| \sigma_g ((x_n)^h) - (x_n)^h \|_2^2 + 2\| P_{\rho(g) \mathcal{X}_h^i} ((x_n)^h) \|_2^2 \right) \\
& =  2 \sum_{g \in \mathcal G} \| \sigma_g ((x_n)^h) - (x_n)^h \|_2^2 + 2\sum_{g \in \mathcal{G}} \| P_{\rho(g) \mathcal{X}_h^i} ((x_n)^h) \|_2^2  \\
& \leq  2 \sum_{g \in \mathcal G} \| \sigma_g ((x_n)^h) - (x_n)^h \|_2^2 + 2 \prod_{\ell = 0}^{k - 1} (1 + \delta^{\circ \ell}(\varepsilon))^2 \|(x_n)^h\|_2^2 \\
& \leq  2 \sum_{g \in \mathcal G} \| \sigma_g ((x_n)^h) - (x_n)^h \|_2^2 + 4 \|(x_n)^h\|_2^2.
\end{aligned}
$$
Finally, since $G$ is abelian, summing up over all $h \in G$ and all $i \in \{0, 1, 2, 3\}$, we get
$$2^k \|P_{(K \otimes \ell^2(G)) \oplus \mathscr X} (x_n)\|_2^2 \leq 8 \sum_{g \in \mathcal G} \|u_g x_n u_g^* - x_n\|_2^2 + 16 \|x_n\|_2^2.$$
This yields $\lim_{n \to \omega} \|P_{(K \otimes \ell^2(G)) \oplus \mathscr X} (x_n)\|_2^2 \leq 2^{4 - k}$. Since this is true for every $k \geq 1$, we finally get $\lim_{n \to \omega} \|P_{(K \otimes \ell^2(G)) \oplus \mathscr X} (x_n)\|_2 = 0$. 
\end{proof}

\begin{step}\label{step2}
We have 
$$a \left( (K^\perp \otimes \ell^2(G)) \oplus \mathscr{Y} \right) \perp \mathcal J b^* \mathcal J \mathscr{Y} \mbox{ and }  a \mathscr{Y} \perp \mathcal J b^* \mathcal J \left( (K^\perp \otimes \ell^2(G)) \oplus \mathscr{Y} \right)$$ in the Hilbert space $\mathcal H \otimes \ell^2(G)$. 
\end{step}

\begin{proof}[Proof of Step $\ref{step2}$]
We first prove that $a \left( (K^\perp \otimes \ell^2(G)) \oplus \mathscr{Y} \right) \perp \mathcal J b^* \mathcal J \mathscr{Y}$. Recall that $a = W(\xi_1 \otimes \cdots \otimes \xi_s)$ and $b = W(\eta_1 \otimes \cdots \otimes \eta_t)$ with $\xi_i, \eta_j \in K$. Using the Fourier decomposition, it suffices to show that for all $h \in G$, $a (K^\perp  \oplus \mathcal Y_h) \perp J \sigma_h(b)^* J \mathcal Y_h$ in the Hilbert space $\mathcal H$. 

Let $e_1 \otimes \cdots \otimes e_m$ be an elementary word in $K^\perp  \oplus \mathcal Y_h$ with $e_1 \in K^\perp$ (possibly $m = 1$). Let $f_1 \otimes \cdots \otimes f_n$ be an elementary word in $\mathcal Y_h$ with $n \geq 2$, $f_1 \in K^\perp$ and $f_n \in (\pi(h)K)^\perp$. Proposition $\ref{wick}$ yields 
$$
\begin{aligned}
a (e_1 \otimes \cdots \otimes e_m) & = \xi_1 \otimes \cdots \otimes \xi_s \otimes e_1 \otimes \cdots \otimes e_m \\
J \sigma_h(b)^* J (f_1 \otimes \cdots \otimes f_n) & = f_1 \otimes \cdots \otimes f_n \otimes \pi(h)\eta_1 \otimes \cdots \otimes \pi(h)\eta_t.
\end{aligned}
$$
Since $\xi_1 \in K$ and $f_1 \in K^\perp$, we get $a (e_1 \otimes \cdots \otimes e_m) \perp J \sigma_h(b)^* J (f_1 \otimes \cdots \otimes f_n)$. This shows that $a (K^\perp  \oplus \mathcal Y_h) \perp J \sigma_h(b)^* J \mathcal Y_h$ in the Hilbert space $\mathcal H$. 

Since $K = \overline K$, $a^*$ and $b^*$ have all their letters in $K$ and the above proof shows that $\mathcal J b \mathcal J \mathscr{Y} \perp a^* \left( (K^\perp \otimes \ell^2(G)) \oplus \mathscr{Y} \right)$. Since $a$ and $\mathcal J b \mathcal J$ commute, we finally obtain that $a \mathscr{Y} \perp \mathcal J b^* \mathcal J \left( (K^\perp \otimes \ell^2(G)) \oplus \mathscr{Y} \right)$.
\end{proof}

\begin{step}\label{step3}
Let $x, y \in (M^\omega \ominus \LL(G)^\omega) \cap \LL(G)'$. Then we have 
$$\lim_{n \to \omega} \langle a P_{K^\perp \otimes \ell^2(G)}(x_n), \mathcal J b^* \mathcal J P_{K^\perp \otimes \ell^2(G)}(y_n) \rangle = 0.$$
\end{step}

\begin{proof}[Proof of Step $\ref{step3}$]
Let $x, y \in (M^\omega \ominus \LL(G)^\omega) \cap \LL(G)'$. Write $P_{K^\perp \otimes \ell^2(G)}(x_n) = \sum_{h \in G} W(e_n^h)\Omega \otimes \delta_h$ and $P_{K^\perp \otimes \ell^2(G)}(y_n) = \sum_{h \in G} W(f_n^h) \Omega \otimes \delta_h$ with $e_n^h, f_n^h \in K^\perp$. Note that $\sum_{h \in G} \|e_n^h\|^2 \leq \|x_n\|_2^2$ and $\sum_{h \in G} \|f_n^h\|^2 \leq \|y_n\|_2^2$. Using Proposition $\ref{wick}$, a simple calculation shows that 
$$
\begin{aligned}
a P_{K^\perp \otimes \ell^2(G)}(x_n) & = \sum_{h \in G} W(\xi_1 \otimes \cdots \otimes \xi_s \otimes e_n^h) \Omega \otimes \delta_h \\
\mathcal J b^* \mathcal J P_{K^\perp \otimes \ell^2(G)}(y_n) & = \sum_{h \in G} W(f_n^h) W(\pi(h)\eta_1 \otimes \cdots \otimes \pi(h)\eta_t) \Omega \otimes \delta_h.
\end{aligned}
$$

Put $A_n := \langle a P_{K^\perp \otimes \ell^2(G)}(x_n), \mathcal J b^* \mathcal J P_{K^\perp \otimes \ell^2(G)}(y_n) \rangle$. Therefore, using again Proposition $\ref{wick}$, we obtain
$$
\begin{aligned}
A_n & = \sum_{h \in G} \langle W(\xi_1 \otimes \cdots \otimes \xi_s \otimes e_n^h)\Omega,  W(f_n^h) W(\pi(h)\eta_1 \otimes \cdots \otimes \pi(h)\eta_t)\Omega \rangle \\
& = \sum_{h \in G} \langle W(f_n^h)^* W(\xi_1 \otimes \cdots \otimes \xi_s \otimes e_n^h)\Omega, W(\pi(h)\eta_1 \otimes \cdots \otimes \pi(h)\eta_t)\Omega \rangle \\
& = \sum_{h \in G} \langle W(\overline f_n^h \otimes \xi_1 \otimes \cdots \otimes \xi_s \otimes e_n^h)\Omega, W(\pi(h)\eta_1 \otimes \cdots \otimes \pi(h)\eta_t)\Omega \rangle \\
& = \sum_{h \in G} \langle \overline f_n^h \otimes \xi_1 \otimes \cdots \otimes \xi_s \otimes e_n^h, \pi(h)\eta_1 \otimes \cdots \otimes \pi(h)\eta_t \rangle
\end{aligned}
$$
Note that if $t \neq s + 2$, then $A_n = 0$ for all $n \in \N$, whence $\lim_{n \to \omega} A_n = 0$.

Next, assume that $t = s + 2$ and fix $\varepsilon > 0$. Since $\pi$ is mixing, there exists a finite subset $\mathcal F \subset G$ such that for all $h \in G \setminus \mathcal F$, we have $|\langle \xi_1, \pi(h)\eta_2 \rangle| \leq \varepsilon$. So, for all $h \in G \setminus \mathcal F$, we have
$$| \langle \overline f_n^h \otimes \xi_1 \otimes \cdots \otimes \xi_s \otimes e_n^h, \pi(h)\eta_1 \otimes \cdots \otimes \pi(h)\eta_t \rangle | \leq \varepsilon \|e_n^h\| \|f_n^h\|.$$
By the Cauchy-Schwarz inequality, for all $n \in \N$, we have
$$\sum_{h \in G \setminus \mathcal F} | \langle \overline f_n^h \otimes \xi_1 \otimes \cdots \otimes \xi_s \otimes e_n^h, \pi(h)\eta_1 \otimes \cdots \otimes \pi(h)\eta_t \rangle | \leq \varepsilon \|x_n\|_2 \| y_n \|_2 \leq \varepsilon.$$
For all $g \in G$, since $\rho(g)(H \otimes \ell^2(G)) = H \otimes \ell^2(G)$, we have
$$
\begin{aligned}
\sum_{h \in G} (\pi(g)e_n^h - e_n^h) \otimes \delta_h & = \rho(g)P_{K^\perp \otimes \ell^2(G)}(x_n) - P_{K^\perp \otimes \ell^2(G)}(x_n) \\
& = P_{H \otimes \ell^2(G)}(u_g x_n u_g^* - x_n ) + (1 - \rho(g))P_{K \otimes \ell^2(G)}(x_n).
\end{aligned}
$$
Using the fact that $x, y \in (M^\omega \ominus \LL(G)^\omega) \cap \LL(G)'$ together with Step $\ref{step1}$, we get that for all $g, h \in G$, $\lim_{n \to \omega} \|\pi(g)e_n^h - e_n^h\| = 0$ and $\lim_{n \to \omega} \|\pi(g)f_n^h - f_n^h\| = 0$. Since $\pi$ is mixing and thus ergodic, we get that $e_n^h \to 0$ and $f_n^h \to 0$ weakly in $H$ as $n \to \omega$ for all $h \in G$. Since $\mathcal F$ is finite, this implies 
$$\lim_{n \to \omega} \sum_{h \in \mathcal F} | \langle \overline f_n^h \otimes \xi_1 \otimes \cdots \otimes \xi_s \otimes e_n^h, \pi(h)\eta_1 \otimes \cdots \otimes \pi(h)\eta_t \rangle | = 0.$$
Thus we have $\lim_{n \to \omega} | A_n | \leq \varepsilon$. Since this is true for every $\varepsilon > 0$, we get $\lim_{n \to \omega} A_n = 0$.
\end{proof}

Let $x, y \in (M^\omega \ominus \LL(G)^\omega) \cap \LL(G)'$. By combining Steps $\ref{step2}$ and $\ref{step3}$, we obtain 
$$\lim_{n \to \omega} \langle a P_{(K^\perp \otimes \ell^2(G)) \oplus \mathscr{Y}}(x_n), \mathcal J b^* \mathcal J P_{(K^\perp \otimes \ell^2(G)) \oplus \mathscr{Y}}(y_n) \rangle = 0.$$
Moreover, Step $\ref{step1}$ yields
$$\lim_{n \to \omega} \|P_{(K \otimes \ell^2(G)) \oplus \mathscr{X}}(x_n)\|_2 = 0 \mbox{ and } \lim_{n \to \omega} \|P_{(K \otimes \ell^2(G)) \oplus \mathscr{X}}(y_n)\|_2 = 0.$$
Therefore, thanks to Equality $(\ref{module})$, we finally get
$$\langle a x, y b\rangle_{\LL^2(M^\omega)} = \lim_{n \to \omega} \langle a x_n, y_n b \rangle_{\LL^2(M)} = \lim_{n \to \omega} \langle a x_n, \mathcal J b^* \mathcal J y_n \rangle_{\LL^2(M)} = 0.$$
As we mentioned before, this finishes the proof.
\end{proof}

\section{Proof of the Theorem and the Corollary}\label{proof}

We prove a stronger version of the main Theorem.

\begin{theo}
Let $G$ be a countable infinite abelian group and $\pi : G \to \mathcal O(H_\R)$ a faithful mixing orthogonal representation. Then for any intermediate von Neumann subalgebra $\LL(G) \subset P \subset \Gamma(H_\R)\dpr \rtimes_\pi G$, there exist pairwise orthogonal projections $p_n \in \mathcal Z(P)$ with $\sum_{n \geq 0} p_n = 1$ such that
\begin{itemize}
\item $P p_0 = \LL(G) p_0$ and
\item $P p_n$ is a non-Gamma ${\rm II_1}$ factor for all $n \geq 1$.
\end{itemize} 
\end{theo}

\begin{proof}
Put $A = \LL(G)$ and $M = \Gamma(H_\R)'' \rtimes G$. Since $A \subset M$ is a masa, we have $\mathcal Z(P) \subset A$. Denote by $p \in \mathcal Z(P)$ the maximal projection such that $P p$ is amenable. Then $P(1 - p)$ has no amenable direct summand. 

By \cite[Theorem 3.10 and Theorem 5.1]{houdayer-shlyakhtenko}, $M$ is a strongly solid ${\rm II_1}$ factor and in particular {\em solid} in the sense of \cite{ozawa}, that is, the relative commutant of any diffuse subalgebra of $M$ must be amenable. (The strong solidity result \cite[Theorem 3.10]{houdayer-shlyakhtenko} is only stated for mixing orthogonal representations of $\Z$ but the same proof works for any countable infinite abelian group $G$ as well). Since $P(1 - p)$ has no amenable direct summand and is solid, we get that its center $\mathcal Z(P(1 - p)) = \mathcal Z(P) (1 - p)$ is purely atomic. Denote by $p_n$, $n \geq 1$, the minimal projections of $\mathcal Z(P)(1 - p)$. For every $n \geq 1$, since the ${\rm II_1}$ factor $P p_n$ is solid and nonamenable, it does not have property Gamma by \cite[Proposition 7]{ozawa}.

It remains to prove that $A p = P p$. The rest of the proof is now identical to the one of \cite[Corollary 2.3]{CFRW}, but we nevertheless give a detailed proof for the sake of completeness. We first show that $P p$ is of type ${\rm I}$. Indeed, assume by contradiction that there exists a nonzero projection $q \in \mathcal Z(P) p$ such that $Pq$ is of type ${\rm II_1}$. Since $P q$ is hyperfinite by Connes' result \cite{connes76}, we may find an increasing sequence $Q_k \subset Pq$ of finite dimensional unital $\ast$-subalgebras such that $\bigvee_{k \geq 1} Q_k = Pq$. Since $Q_k' \cap Pq$ is of type ${\rm II_1}$ and $A$ is abelian, \cite[Corollary 2.3]{popa-malleable1} yields a unitary $u_k \in \mathcal U(Q_k' \cap Pq)$ such that $\|E_{A}(u_k)\|_2 \leq \frac1k$ for all $k \geq 1$. Therefore, the sequence $(u_k)$ represents a unitary $u \in \mathcal U((Pq)' \cap (Pq)^\omega)$ such that $E_{A^\omega}(u) = 0$. Since $A q \subset P q$ is a masa in a type ${\rm II_1}$ von Neumann algebra, we may find a unitary $v \in \mathcal U(P q)$ such that $E_A(v) = E_{A q}(v) = 0$. By Theorem $\ref{theorem-AOP}$, we get that $v u$ and $u v$ are orthogonal in $\LL^2(M^\omega, \tau_\omega)$. Since moreover $v u = u v$, we obtain $u v = 0$, whence $q = (u v)^* (uv) = 0$. This is a contradiction. Therefore $P p$ is of type ${\rm I}$. Since $A \subset P p \oplus A (1 - p)$ is a masa in a finite type ${\rm I}$ von Neumann algebra, we have that $A$ is regular inside $P p \oplus A(1 - p)$ by \cite[Theorem 3.19]{kadison}. By singularity of $A$, we get $A = P p \oplus A(1 - p)$ and so $A p = P p$.
\end{proof}

\begin{proof}[Proof of the Corollary]
The proof is very similar to the one of \cite[Theorem 5.7]{kechris}. For $\mu \in \Prob(\mathbf T)$ a Borel probability measure on $\mathbf T$, we use the notation $\mu^\infty = \sum_{n \geq 1}\frac{1}{2^n}  \mu^{\ast n}$. Write $\supp(\mu)$ for the {\em topological support} of $\mu$, that is, 
$$\supp(\mu) = \bigcap \{F \subset \mathbf T \mbox{ closed subset} : \mu(F) = 1\}.$$
We have $\supp(\mu \ast \nu) \subset \supp(\mu) \supp(\nu)$ for all $\mu, \nu \in \Prob(\mathbf T)$. Define the real Hilbert space 
$$H_\R^\mu = \{ \zeta \in \LL^2(\mathbf T, \mu) : \overline{\zeta(z)} = \zeta(\overline z) \; \mu\mbox{-almost everywhere}\}$$ 
and the orthogonal representation
$$\pi^\mu : \Z \to \mathcal O(H_\R^\mu) : \left( \pi^\mu(n) \zeta \right)(z) = z^n \zeta(z).$$
Observe that the complexification of $H_\R^\mu$ is simply $\LL^2(\mathbf T, \mu)$. The corresponding unitary representation on $\LL^2(\mathbf T, \mu)$ will still be denoted by $\pi^\mu$.

Using a combination of \cite[${\rm VIII}$, 3, Th\'eor\`eme ${\rm II}$]{kahane-salem} and \cite[${\rm VII}$, 1, Theorem 7]{kechris-louveau}, there exists a closed independent\footnote{For all  distinct elements $z_1, \dots, z_k \in \Lambda$ and all $n_1, \dots, n_k \in \Z$, if $z_1^{n_1} \cdots z_k^{n_k} = 1$ then $n_1 = \cdots = n_k = 0$.} set $\Lambda \subset \mathbf T$ and a Borel map $2^\N \ni x \mapsto \mu_x \in \Prob(\mathbf T)$ such that: 
\begin{itemize}
\item For all $x \in 2^\N$, $\mu_x$ is a symmetric Rajchman measure such that $\supp(\mu_x) \subset \Lambda \cup \overline \Lambda$. 
\item For all $x, y \in 2^\N$ such that $x \neq y$, we have $\supp(\mu_x) \cap \supp(\mu_y) = \emptyset$.
\end{itemize}

Since $\mu_x$ is a Rajchman measure, $\LL(\Z) \subset \Gamma(H_\R^{\mu_x})\dpr \rtimes_{\pi^{\mu_x}} \Z$ is maximal amenable by the theorem. Put 
$$A = \LL(\Z) \mbox{ and } \vphantom{}_A\mathcal H(x)_A = \vphantom{}_A\LL^2((\Gamma(H_\R^{\mu_x})\dpr \rtimes_{\pi^{\mu_x}} \Z) \ominus \LL(\Z))_A.$$ 
We have $\supp(\mu_x^{\ast n}) \subset (\Lambda \cup \overline \Lambda)^n$ for all $x \in 2^\N$ and all $n \geq 1$. Since the measures $(\mu_x)_{x \in 2^\N}$ are atomless with pairwise disjoint supports and $\Lambda$ is a closed independent set, we obtain that the measures $\mu_x^{\ast n}$ for $x \in 2^\N$ and $n \geq 1$ are pairwise singular by \cite[Theorem 5.3.2]{rudin}. In the language of spectral theory, this shows that the maximal spectral type of the unitary representation $\bigoplus_{n \geq 1} (\pi^{\mu_x})^{\otimes n}$ is equal to $\mu_x^\infty$ and that the measures $(\mu_x^\infty)_{x \in 2^\N}$ are moreover pairwise singular. Since $\Lambda$ is a closed independent set, the subgroup $H(\Lambda) \subset \mathbf T$ generated by $\Lambda$ has Haar measure zero by \cite[Theorem 5.3.6]{rudin}. In particular, the measures $\mu_x^\infty$ are singular with respect to the Haar measure for all $x \in 2^\N$.

Write $\mathbf \Psi : \mathbf T^2 \to \mathbf T^2$ for the group homomorphism $\mathbf \Psi(z_1, z_2) = (z_1 \overline z_2, z_2)$. Then the map 
$$\eta : 2^\N \to \Prob(\mathbf T^2) : x \mapsto \eta_x = \mathbf \Psi_\ast(\mu_x^\infty \times \Haar)$$ 
is Borel and the measure class of the $A$-$A$-bimodule $\mathcal H(x)$ is the class of $\eta_x$. Since the measures $(\mu_x^\infty)_{x \in 2^\N}$ are pairwise singular and all singular with respect to the Haar measure on $\mathbf T$, the measures $(\eta_x)_{x \in 2^\N}$ are pairwise singular and all singular with respect to the Haar measure on $\mathbf T^2$. Therefore the $A$-$A$-bimodules $\mathcal H(x)$ are pairwise non-isomorphic and all disjoint from the coarse $A$-$A$-bimodule.  This finishes the proof.
\end{proof}

\bibliographystyle{plain}

\end{document}